\documentclass{amsart}

\usepackage[hidelinks]{hyperref}
\usepackage{amssymb}
\usepackage{mathtools}
\usepackage{enumitem}
\usepackage{algorithm}
\usepackage{algpseudocodex}
\usepackage{color}
\usepackage{aliascnt}
\usepackage{cleveref}
\algtext*{EndWhile}
\algtext*{EndFor}
\algtext*{EndIf}

\allowdisplaybreaks

\renewcommand{\leq}{\leqslant}
\renewcommand{\geq}{\geqslant}

\newcommand{\Z}{\mathbb{Z}}

\newcommand{\eps}{\varepsilon}

\newcommand{\divides}{\mathrel{|}}
\newcommand{\ndivides}{\mathrel{\nmid}}

\DeclareMathOperator{\ord}{ord}
\DeclareMathOperator{\lcm}{lcm}

\makeatletter
\newcommand{\tpmod}[1]{{\@displayfalse\pmod{#1}}}
\makeatother

\makeatletter
\algnewcommand{\LineComment}[1]{\Statex \hskip\ALG@thistlm {\color{gray} -- \emph{#1}}}
\makeatother

\makeatletter
\newcommand{\StateFullWidth}[1]{\State \makebox[\dimexpr\linewidth-\algorithmicindent\relax][l]{#1}}
\makeatother

\newlist{inputoutputlist}{itemize}{1}
\setlist[inputoutputlist]{label=--,topsep=0pt,leftmargin=20pt}

\copyrightinfo{2026}{David Harvey and Markus Hittmeir}

\newtheorem{thm}{Theorem}[section]

\newaliascnt{lem}{thm}
\newtheorem{lem}[lem]{Lemma}
\aliascntresetthe{lem}

\newaliascnt{prop}{thm}

\aliascntresetthe{prop}

\newaliascnt{cor}{thm}

\aliascntresetthe{cor}

\theoremstyle{definition}
\theoremstyle{remark}
\newaliascnt{rem}{thm}
\newtheorem{rem}[rem]{Remark}
\aliascntresetthe{rem}

\numberwithin{equation}{section}
\numberwithin{algorithm}{section}

\begin{document}

\title[Deterministically finding elements of large order]%
{Deterministic methods for finding elements \\ of large multiplicative order}


\author[D. Harvey]{David Harvey}
\address{School of Mathematics and Statistics, University of New South Wales, Sydney NSW 2052, Australia}
\email{d.harvey@unsw.edu.au}

\author[M. Hittmeir]{Markus Hittmeir}
\address{NORCE Research, Nygårdsgaten 112, 5008 Bergen, Norway}
\email{mahi@norceresearch.no}
\thanks{The second author was supported by the Research Council of Norway (grant 357539).}


\date{}

\begin{abstract}
We revisit the problem of rigorously and deterministically finding elements of large order in the multiplicative group of integers modulo a natural number $N$.
Solving this problem is an essential step in several recent deterministic algorithms for factoring $N$, including the currently fastest ones.
In 2018, the second author gave an algorithm that for a given target order $D \geq N^{2/5}$, finds either an element of order exceeding~$D$, or a nontrivial divisor of $N$, or proves that $N$ is prime.
The running time was
\[
O\left(\frac{D^{1/2}}{(\log \log D)^{1/2}} \log^2 N \right)
\]
bit operations, asymptotically the same as the cost of computing the order of a \emph{single} element using Sutherland's optimisation of the classical babystep-giantstep method.
Subsequent work by several authors weakened the hypothesis $D \geq N^{2/5}$ to $D \geq N^{1/6}$. In this paper, we show that the hypothesis may be dropped altogether. Moreover, if $N$ is prime, we can guarantee returning an element of order exceeding~$D$, rather than a proof that $N$ is prime.
\end{abstract}

\maketitle
\section{Introduction}
\label{sec:intro}

Let $N \geq 2$ be an integer and consider
the multiplicative group $\Z_N^*$ of invertible residues modulo $N$.
The order of an element $\alpha \in \Z_N^*$, denoted $\ord_N(\alpha)$,
is defined to be the smallest positive integer $k$ such that $\alpha^k = 1$.

In this paper we are interested in the following \emph{large order problem}:
given $N$ and a positive integer~$D$,
find (if possible) an element of $\Z_N^*$ such that ${\ord_N(\alpha) > D}$.
This problem appears in many practical applications,
including for example cryptographic schemes based on the
discrete logarithm problem \cite{DH-dlp, E-elgamal}
and the generation of pseudorandom numbers \cite{BM-prng}.

For these and other real-world applications it is entirely reasonable
to attack the large order problem via probabilistic and/or heuristic methods.
However, in this paper we are mainly interested in \emph{rigorous},
\emph{deterministic} algorithms,
with our principal motivation being the task of
rigorous and deterministic \emph{integer factorisation}.
We will return to this application in a few moments,
but let us first briefly discuss what can be said
about the large order problem in a wider algorithmic context.

\subsection{Heuristic and randomised algorithms}
\label{sec:randomised}

If our goal is merely to find an element that with \emph{high probability}
has order larger than~$D$,
then we may take a \emph{Monte Carlo} approach and simply
select some $\alpha \in \Z_N^*$ at random.
The success probability depends on the size of~$D$ and the density of
elements with order larger than~$D$.
This procedure is often effective for small values of~$D$,
and even for larger~$D$ when $\phi(N)$ (the order of the group~$\Z_N^*$)
has few small prime divisors.

It is possible to boost the success probability considerably if the
prime factorisation $N = p_1^{e_1} \cdots p_k^{e_k}$ is known.
(This includes the case that $N$ is itself prime.)
Indeed, after computing $\phi(N)$
and finding all of its ``small'' prime divisors~$q$,
we may easily determine the $q$-part of $\ord_N(\alpha)$
for a given $\alpha \in \Z_N^*$ by computing
$\alpha^{\phi(N)/q^j} \pmod{p_i^{e_i}}$ for various~$i$ and~$j$.
By sampling repeatedly,
we can easily find an element $\alpha \in \Z_N^*$
with maximal $q$-power order for all of these~$q$.
For such $\alpha$, it becomes overwhelmingly likely that it has
maximal $q$-power order for the remaining (unknown but large)
prime divisors $q$ of $\phi(N)$.

If however our goal is to return an element of order larger than $D$
with \emph{certainty} (a \emph{Las Vegas} algorithm),
then the key issue is to be able to actually compute the order of elements,
or otherwise certify that the order is large.
If we know the complete prime factorisation of both $N$ and $\phi(N)$,
then we can compute orders of elements in (deterministic) polynomial time.
Thus for example if $N$ is prime,
the probability of a random $\alpha \in \Z_N^*$ being a primitive root
is $\phi(N-1)/(N-1) \gg 1/\log \log N$,
so if we know the factorisation of $N-1$,
we can find a primitive root in randomised polynomial time
by sampling $\alpha \in \Z_N^*$ repeatedly until we succeed.

If we insist on deterministic algorithms,
then we may be able to get away with heuristic assumptions
as a substitute for randomness.
For example, under the Extended Riemann Hypothesis it is known
that for $N$ prime there exists a primitive root less than $(\log N)^{O(1)}$
\cite{W-leastprimroot,S-primroots}.
In this setting we can find a primitive root in deterministic polynomial time,
again assuming that the factorisation of $N-1$ is known.

It is clear from the preceding discussion that if we want certified output,
then regardless of whether we allow randomness or unproved hypotheses,
the real bottleneck is the integer factorisation step.
Of course, no polynomial time factorisation algorithm is known,
even allowing probabilistic or heuristic methods.
Algorithms such as the quadratic sieve and the general number field sieve
\cite{Rie-factorization,Wag-joy}
achieve heuristic runtime complexities subexponential in $\log N$.
The latter can (conjecturally) find the factorisations of $N$ and $\phi(N)$
in expected time
\[
L_N\left[1/3, \left(64/9\right)^{1/3}\right]
=
\exp\left(
\left(\left(64/9\right)^{1/3}+o(1)\right)
(\log N)^{1/3}(\log\log N)^{2/3}
\right).
\]
From a practical point of view,
it is this last bound that determines the difficulty
of the large order problem in its most general form.

\subsection{Rigorous and deterministic algorithms}

To motivate the approach to the large order problem taken in this paper,
we will first explain the crucial role it has played in recent work
on deterministic integer factorisation.

Until recently, the best complexity bounds for rigorously
factoring an integer~$N$ were of the form $N^{1/4+o(1)}$,
based on methods going back to Pollard \cite{Pol-factoring},
Strassen \cite{Str-factoring} and Coppersmith \cite{Cop-lowexponent}.
(Here, and for the rest of the paper,
time complexity always refers to the number of steps performed by a
deterministic multitape Turing machine.
For a detailed description of this computational model,
see \cite[Ch.\,2]{Pap-complexity} or \cite[\S1.6]{AHU-algorithms}.)
The first improvement on these bounds by a superpolynomial factor
(i.e., larger than any fixed power of $\log N$) appeared in \cite{Hit-BSGS},
and was based on the idea of solving the discrete logarithm problem
$\alpha^x = \alpha^{N+1}$ in $\Z_N^*$ for a given element $\alpha$.
If $N$ is a product of two distinct primes, say $N=pq$,
then the solution $x=p+q$ immediately reveals the two factors of $N$.
However, we can only find this solution if $\ord_N(\alpha)$
is sufficiently large.
Therefore, the paper also included an algorithm that,
for $D\geq N^{2/5}$, either finds an element of order at least~$D$,
or a nontrivial divisor of $N$, or proves $N$ prime \cite[Alg.\,6.2]{Hit-BSGS}.
Notice that this algorithm does not always actually
solve the large order problem for $(N,D)$;
if $N$ is composite, it is allowed to
return a nontrivial divisor of $N$ instead.
This is of course entirely reasonable in the context
of a factorisation algorithm.

The algorithm has two stages.
In the first stage, we search for the order of small elements
$\beta=2,3,\ldots$ via Sutherland's optimisation of
the classical babystep-giantstep method
\cite{Sha-classnumber,SutherlandPHD},
assuming that the order is at most $D$.
If for some element we fail to compute the order,
then the order exceeds $D$ and we are done.
If the order $m \coloneqq \ord_N(\beta)$ is actually found,
we try to use it to compute a nontrivial divisor of $N$,
by examining $\gcd(\beta^{m/r} - 1, N)$ for prime divisors $r$ of $m$.
If that fails, we deduce that $m$ divides $p-1$ for
every prime factor $p$ of $N$.
As we compute the exact orders of these small elements,
and continue to fail to factorise $N$,
we accumulate considerable information about the prime factors of $N$.
Eventually, we recover sufficient information to proceed to
the second stage of the algorithm,
in which we attempt to factorise $N$ directly.
If that fails too, we obtain a proof that $N$ is prime.
It is noteworthy that the time complexity of this procedure is
asymptotically the same as testing whether a single element
has order at most $D$, namely
\begin{equation}
\label{eq:previous-bound}
O\left(\frac{D^{1/2}}{(\log \log D)^{1/2}} \log^2 N\right).
\end{equation}

The 2018 algorithm for finding elements of large order
was later used as a subroutine in a series of works
\cite{Hit-timespace,Har-onefifth,HH-onefifthloglog}
that improved the runtime of deterministic factorisation to $N^{1/5+o(1)}$.
The original assumption $D \geq N^{2/5}$ was good enough
to directly apply the original algorithm in these factorisation methods,
but only just.
The hypothesis on $D$ was subsequently improved on two occasions,
first to $D \geq N^{1/4+o(1)}$ by Gao, Feng, Hu and Pan
\cite[Lem.\,3.5]{GFHP-factoringlattices}
and then to $D \geq N^{1/6}$ by Oznovich and Volk
\cite[Thm.\,1.1]{OV-highorder}.
These improvements are based on a better bound for the number of
consecutive elements with the same order $m$
(namely, $O(\sqrt{m})$ instead of $m$),
leading to a faster progression in the loop over the elements
$\beta = 2,3,\ldots$ in the first stage of the algorithm.
Both of these papers (and additionally \cite{HH-lehmangeneralization})
also considered the case of $N$ having an $r$-power divisor,
which allows the assumption to be relaxed further to $D \geq N^{1/6r}$
\cite[Thm.\,4.2]{OV-highorder}.
Furthermore, both \cite{GFHP-factoringlattices} and \cite{OV-highorder}
present improved techniques for the second stage,
in which $N$ is factorised based on a large known factor of $p-1$
for all primes $p \divides N$. 

All of the results mentioned so far impose a fairly restrictive
(in fact, fully exponential) lower bound hypothesis on $D$.
The main result of this paper is an algorithm that solves
the same problem in essentially the same amount of time,
but without any restriction on $D$ at all.
\begin{thm}
\label{thm:main}
There exists an algorithm with the following properties. 
It takes as input integers $N \geq 3$ and $D \geq 1$ with $D < N-1$.
It outputs either some $\alpha\in\Z_N^*$ with $\ord_N(\alpha) > D$
or a nontrivial divisor of $N$.
Its time complexity is
\begin{equation}
\label{eq:main-complexity}
O\left(\frac{D^{1/2} \log D}{(\log \log D)^{1/2}} \log N \right).
\end{equation}
\end{thm}

Note that the expression $\log \log D$ does not literally make sense for $D \leq 2$; in cases like this, the reader should mentally substitute $\log \log \max(D, 3)$.
It can be shown that the space complexity in \Cref{thm:main} is given by
\[
O\left(\frac{D^{1/2}}{(\log \log D)^{1/2}} \log N\right),
\]
but we will not carry out the details of this analysis.

The reader will notice a small improvement in time complexity
between \eqref{eq:previous-bound} and \eqref{eq:main-complexity},
namely that one factor of $\log N$ has been replaced by $\log D$.
In the first version of this paper,
\Cref{thm:main} was stated with the previous bound \eqref{eq:previous-bound}
instead of \eqref{eq:main-complexity}.
An anonymous referee subsequently pointed out that the analysis of
Sutherland's order-finding algorithm could be improved slightly in the Turing model (see \Cref{lem:bsgs}),
leading to the tighter bound \eqref{eq:main-complexity}.

We point out several interesting features of this result:
\begin{itemize}
\item
Unlike its predecessors, this algorithm never returns ``$N$ is prime''.
In particular, if $N$ is prime, it \emph{always} returns
$\alpha \in \Z_N^*$ with $\ord_N(\alpha) > D$
(and thus solves the large order problem as originally stated).
This is interesting from a theoretical point of view,
as it permits rigorously finding not only primitive roots
but also elements of order exceeding $D$, even for small values of~$D$.
(When $D$ is sufficiently large, namely for $D \geq N^{1/2+\eps}$,
it was already known how to find elements of order at least $D$
in time $O(D^{1/2+\eps})$, as one can find a primitive root in time
$O(N^{1/4+\eps})$ \cite{Shp-primitive}.)
It is also worth mentioning that when $D$ is very small,
the runtime may be even less than the cost of testing $N$ for primality
via a deterministic polynomial time primality test,
such as the AKS test \cite{AKS-primes}.

\item
An important consequence of the theorem is that in the context of
\emph{any} deterministic factoring algorithm that runs in exponential time,
finding elements of large order should no longer be considered a bottleneck,
regardless of the exponent.

\item
\Cref{thm:main} outperforms the (heuristic) approach mentioned
in \Cref{sec:randomised} whenever
\[
D\ll L_N\left[1/3, 2(64/9)^{1/3}\right].
\]
For such inputs~$D$, we are not aware of a general-purpose
Las Vegas algorithm that,
without obtaining and using the prime factorisation of~$N$,
certifies $\ord_N(\alpha)>D$ faster than the
deterministic bound given in \eqref{eq:main-complexity}.
\end{itemize}

The main idea behind the improved algorithm is based on estimates for the density of integers that factor completely into primes $q \leq B$ for a given bound $B$, the so-called \emph{$B$-smooth numbers}.
The algorithm has a broadly similar structure to its predecessors.
In the first stage, we attempt to compute the order of all elements $\beta = 2,3,\ldots, B$, for a suitable choice of $B$.
If we have not yet found an element of large order,
or a factor of $N$ via GCDs as before,
then we have computed some $M \leq D$ such that all these elements satisfy the congruence $\beta^M \equiv 1 \pmod p$ for every prime factor $p$ of $N$, and such that $M \divides p-1$ for every such~$p$.
The key observation is that not only these $\beta$, but \emph{all $B$-smooth integers} in the interval $1 \leq n \leq p$ satisfy $n^M \equiv 1 \pmod p$.
But the latter congruence has at most $M$ solutions modulo $p$,
so by applying well-known lower bounds for the density of $B$-smooth numbers, we obtain drastically improved estimates for how fast $M$ grows as we progress through the main loop over $\beta$. This strategy incidentally leads to a considerable simplification in the second stage of the algorithm: it suffices to use trial division to check for divisors along the arithmetic progression $kM+1$ instead of applying more elaborate techniques.

The rest of the paper is structured as follows. \Cref{sec:prelim} discusses preliminaries such as the babystep-giantstep method for finding orders of elements, and bounds for the density of $B$-smooth numbers. Then the main theorem is proved in \Cref{sec:main}.

\begin{rem}
After the initial version of our paper was made public,
we became aware of independent concurrent work by Itamar Nir,
which was subsequently posted on arXiv \cite{Nir2026}.
Nir obtains a result similar to but weaker than our \Cref{thm:main},
the main difference being that it requires the \emph{subexponential}
bound $D > \exp(\sqrt{2 \log N \log \log N})$.
His argument makes similar use of smoothness bounds,
and according to Nir is simpler than our proof of \Cref{thm:main}.
\end{rem}

\subsection*{Acknowledgments}

The authors would like to thank the anonymous referee mentioned earlier,
whose comments led to the tightening of the bound in \Cref{thm:main},
and also the other anonymous referees whose suggestions helped to
improve the presentation of these results.

\section{Preliminaries}
\label{sec:prelim}
We sometimes use the notation $f \ll g$ as a synonym for $f = O(g)$. The notation $f \asymp g$ means that both $f \ll g$ and $g \ll f$ hold.

We will frequently use standard results on the complexity of integer arithmetic:
we may compute sums and differences of $n$-bit integers in time $O(n)$,
products in time $O(n \log n)$ \cite{HvdH-nlogn},
quotients and remainders in time $O(n \log n)$ \cite[Ch.\,9]{vzGG-compalg3},
$k$-th roots with remainder for fixed $k$
in time $O(n \log n)$ \cite[\S1.5]{BZ-mca},
and greatest common divisors (including finding the corresponding cofactors in the B\'ezout identity) in time $O(n \log^2 n)$ \cite[Ch.\,11]{vzGG-compalg3}.

Let $N \geq 2$. Elements of the ring $\Z_N$ are always represented by their residue in the interval $[0, N)$. Given $\alpha \in \Z_N$ and an integer $n \geq 1$, we may compute $\alpha^n$ in time $O(\log n \log N \log \log N)$ using the repeated squaring method \cite[\S4.3]{vzGG-compalg3}.

We will occasionally need to compute approximations
for logarithms and exponentials of real numbers.
We use standard algorithms to approximate these functions to $n$-bit precision
in time $O(n \log^2 n)$ \cite[\S4.8]{BZ-mca},
omitting the straightforward but tedious analysis of rounding errors.

Whenever we need to store the prime factorisation of a positive integer $n$, say $n = \prod_i q_i^{e_i}$, we always represent it by the list of pairs $(q_i,e_i)$, with the $q_i$ in increasing order and all $e_i \geq 1$. This representation occupies space $O(\log n)$.

At the heart of all our algorithms is a minor variant of Sutherland's well-known method for computing the order of an element of a group.
\begin{lem}
\label{lem:bsgs}
Given as input integers $N \geq 2$, $D \geq 1$ and an element $\alpha \in \Z_N^*$, we may determine if $\ord_N(\alpha) \leq D$,
and if so compute the exact order $m \geq 1$, in time
\[
O\left(\frac{r^{1/2}}{(\log \log r)^{1/2}} (\log r + \log \log N)\log N\right),
\qquad \text{where } r \coloneqq \min(m,D).
\]
\end{lem}
\begin{proof}
According to \cite[Alg.\,4.1]{SutherlandPHD},
we may search for the order of $\alpha$ up to a given $T \geq 1$
by performing $O(T^{1/2}(\log\log T)^{-1/2})$ group operations in $\Z_N^*$
plus a similar number of reads/writes to a lookup table, e.g., a hash table.
One cannot implement hash tables efficiently on a Turing machine,
but it is shown in \cite[Thm.\,6.1]{Hit-BSGS}
(see also \cite[Rem.\,2.5]{Hit-BSGS}) that by replacing the use of
lookup tables by sorted lists and the merge sort algorithm,
one can solve this problem in
\[
O\left(\frac{T^{1/2}}{(\log \log T)^{1/2}} \log^2 N\right)
\]
bit operations on a Turing machine.

As pointed out by an anonymous referee,
this bound can be improved slightly in the following way.
Let $n \coloneqq T^{1/2}(\log\log T)^{-1/2}$.
We may perform $O(n)$ group operations in $\Z_N^*$ using
$O(n \log N \log \log N)$ bit operations,
and we may sort a list of $O(n)$ elements of $\Z_N^*$ using
$O(n \log N \log n) = O(n \log N \log T)$ bit operations.
Inserting these estimates into the argument in \cite[Thm.\,6.1]{Hit-BSGS},
we conclude that we may search for the order of $\alpha$ up to $T$ within
\[
O\left(\frac{T^{1/2}}{(\log \log T)^{1/2}} (\log T + \log \log N) \log N\right)
\]
bit operations.

We apply this result successively for $T=1,2,4,\ldots, 2^{\lceil\log_2 D\rceil}$, terminating immediately if we find $m$. If $m$ is not found, we must have $m > 2^{\lceil\log_2 D\rceil} \geq D$, and we return ``$\ord_N(\alpha) > D$''. The overall time complexity bound follows from the estimate
\begin{align*}
\sum_{i=0}^{\lceil \log_2 r\rceil} 2^{i/2}(\log\log(2^i))^{-1/2} & \ll \sum_{i=0}^{\lceil (\log_2 r)/2\rceil}2^{i/2}+\sum_{i=\lceil (\log_2 r)/2\rceil+1}^{\lceil \log_2 r\rceil}2^{i/2}(\log i)^{-1/2} \\
& \ll r^{1/4} + r^{1/2}(\log\log r)^{-1/2} \\
& \ll r^{1/2}(\log\log r)^{-1/2}. \qedhere
\end{align*}
\end{proof}

\begin{lem}
\label{lem:compord}
Let $N \geq 2$.
Given as input elements $\alpha, \beta \in \Z_N^*$,
and their orders $u \coloneqq \ord_N(\alpha)$ and $v \coloneqq \ord_N(\beta)$,
together with the prime factorisations of $u$ and~$v$,
we may compute an element $\gamma \in \Z_N^*$ of order $w \coloneqq \lcm(u,v)$,
together with the prime factorisation of $w$,
in time
\[
O((\log u + \log v) \log N \log \log N).
\]
\end{lem}

\begin{proof}
Let $u=\prod_i q_i^{e_i}$ and $v=\prod_i q_i^{f_i}$ be the prime factorisations of $u$ and $v$ (temporarily allowing $e_i = 0$ and $f_i = 0$ so that we can use the same list of primes for both $u$ and $v$).
Set $\tilde \alpha \coloneqq \alpha^s$ and $\tilde \beta \coloneqq \beta^t$ where
\[
s \coloneqq \prod_{e_i < f_i} q_i^{e_i}, \qquad
t \coloneqq \prod_{e_i \geq f_i} q_i^{f_i}.
\]
Then $\ord_N(\tilde\alpha) = \prod_{e_i\geq f_i} q_i^{e_i}$ and $\ord_N(\tilde\beta) = \prod_{e_i<f_i} q_i^{f_i}$.
Note that $\ord_N(\tilde\alpha)$ and $\ord_N(\tilde\beta)$ are coprime and their product is equal to $w = \lcm(u,v)$.
Hence $\gamma \coloneqq \tilde\alpha \tilde\beta$ meets our requirements.

Computing $\tilde\alpha = \alpha^{s}$ via the standard binary powering algorithm requires time 
\[
O(\log s \log N \log \log N) = O(\log u \log N \log \log N).
\]
This bound also covers the cost of computing $s$ itself, which requires $O(\log u)$ multiplications 
of bit size $O(\log N)$. The analysis for $\tilde\beta$ is similar.
\end{proof}

\begin{lem}
\label{lem:porder}
Let $N \geq 2$ and $\beta \in \Z_N^*$, and put $m \coloneqq \ord_N(\beta)$. Then
\[
\ord_p(\beta) = m \qquad \text{for all primes $p$ dividing $N$}
\]
if and only if
\[
\gcd(\beta^{m/r} - 1, N) = 1 \qquad \text{for all primes $r$ dividing $m$}.
\]
\end{lem}
\begin{proof}
This is \cite[Lem.\,2.3]{Hit-BSGS}.
\end{proof}

Recall that a positive integer $n$ is said to be \emph{$y$-smooth} if every prime factor $q$ of~$n$ satisfies $q \leq y$. For $x \geq 1$ we denote by $\Psi(x,y)$ the number of positive integers $n \leq x$ that are $y$-smooth. We will need the following lower bound for $\Psi(x,y)$.

\begin{lem}
\label{lem:smooth}
For $x\geq y\geq 2$ and $x\geq 4$ we have
\[
\Psi(x,y)\geq \frac{x}{(\log x)^{\log x / \log y}}.
\]
\end{lem}
\begin{proof}
This is \cite[\S2, Thm.\,1]{Konyagin2013}.
\end{proof}

For any prime $p$, let $g(p)$ denote the least primitive root modulo $p$.
\begin{lem}
\label{lem:burgess}
For any $\eps > 0$ we have $g(p) \ll p^{1/4+\eps}$.
\end{lem}
\begin{proof}
See \cite{B-leastprimitiveroot}.
\end{proof}

\section{Proof of the main theorem}
\label{sec:main}

In this section we prove \Cref{thm:main}.
Recall that the input consists of integers $N \geq 3$ and $D \geq 1$ with $D < N-1$.
Pseudocode for the algorithm is given in \Cref{alg:largeorder}.
In Section \ref{sec:correctness} we prove that the output of the algorithm is correct, and in Section \ref{sec:complexity} we analyse its complexity.

\begin{algorithm}
\caption{The main algorithm}

\raggedright
\textbf{Input}: Integers $N \geq 3$ and $D \geq 1$ with $D < N-1$.

\textbf{Output}: Either some $\alpha \in \Z_N^*$ with $\ord_N(\alpha) > D$, or a nontrivial divisor $d \divides N$.
	
\begin{algorithmic}[1]

\LineComment{See proof of \Cref{thm:main} for the definition of the constant $N_0$.}

\If{$N < N_0$} return suitable $\alpha$ or $d$ via a lookup table. \EndIf
\label{line:lookup}

\If{$N$ is even} return $d \coloneqq 2$. \EndIf
\label{line:N-even}

\If{$2^D < N$} return $\alpha \coloneqq 2$. \EndIf
\label{line:D-small}

\State{Set $B \coloneqq \lceil D^{1/3} \rceil$.}
\label{line:compute-B}

\vspace{5pt}

\LineComment{On entry to each iteration of the for-loop below, $M$ is equal to $\ord_N(\alpha)$.}

\LineComment{The factorisation of $M$ is always known and stored alongside $M$.}

\State Set $\alpha \coloneqq 1$ and $M \coloneqq 1$.
\label{line:alpha-M-init}

\For{$\beta = 2, 3, \ldots, B$}
\label{line:for-loop}

\If{$\beta \divides N$} return $d \coloneqq \beta$. \EndIf
\label{line:beta-divides-N}

\If{$\beta^M \equiv 1 \pmod N$} go to line \ref{line:for-loop} (proceed to next value of $\beta$). \EndIf
\label{line:beta-to-M}

\State Use \Cref{lem:bsgs} to search for $m \coloneqq \ord_N(\beta)$ up to $D$.
\label{line:m-search}

\If{$m > D$} return $\alpha \coloneqq \beta$. \EndIf
\label{line:m-large}

\State Compute the prime factorisation of $m$.
\label{line:m-factor}

\For{each prime $r \divides m$}
	\label{line:r-loop}
	\If{$d \coloneqq \gcd(\beta^{m/r} - 1, N) \neq 1$} return $d$. \EndIf
	\label{line:gcds}
\EndFor

\State Apply \Cref{lem:compord} to $\alpha$ and $\beta$ to compute $\alpha'$ of order $M' \coloneqq \lcm(m, M)$.
\label{line:lcm}

\State Set $\alpha \coloneqq \alpha'$ and $M \coloneqq M'$.
\label{line:alpha-M-update}

\If{$M > D$} return $\alpha$. \EndIf
\label{line:M-large}
\EndFor

\vspace{5pt}
\StateFullWidth{Compute an integer $Z$ such that $\tilde Z < Z < \tilde Z + 2$ where}
\label{line:compute-Z}
\[
\tilde Z \coloneqq 2M \cdot (\log 2M)^{(\log 2M) / (-1 + \log B)}.
\vspace{5pt}
\]

\For{$k = 1, 2, \ldots, \lfloor Z/M \rfloor$}
\label{line:k-loop}

\If{$d \coloneqq kM + 1 \divides N$} return $d$ \EndIf
\label{line:d-test}
\EndFor

\LineComment{The proof of \Cref{thm:main} shows that this line is never reached.}

\end{algorithmic}
\label{alg:largeorder}
\end{algorithm}

In what follows, various arguments will require $N$ or $D$ to be ``sufficiently large''.
We assume that $N_0$ and $D_0$ are absolute constants chosen so that these arguments are valid for all $N \geq N_0$ and $D \geq D_0$.
Note that after line \ref{line:lookup} we have $N \geq N_0$.
Assuming we choose $N_0 \geq 2^{D_0}$,
then after line \ref{line:D-small} we also have $D \geq \log_2 N \geq D_0$.

\subsection{Proof of correctness}
\label{sec:correctness}

A few cases are easily dealt with:
\begin{itemize}
\item
If the algorithm terminates in lines \ref{line:lookup}--\ref{line:D-small}, the output is clearly correct.
(In line \ref{line:lookup}, if $N$ is composite, we return any nontrivial divisor, and if $N$ is prime, we return a primitive root modulo $N$, which has order $N-1 > D$.)
\item
For sufficiently large $N$ and $D$ we have $B < D^{1/3} + 1 < N^{1/3} + 1 < N$.
Hence if the algorithm terminates in line \ref{line:beta-divides-N}, the output is correct.
\item
If the algorithm terminates in line \ref{line:m-large}, the output is clearly correct.
\item
The GCD in line \ref{line:gcds} can never be equal to $N$,
because the definition of $m$ ensures that $\beta^{m/r} \not\equiv 1 \pmod N$.
Thus if the algorithm terminates in this line, the output is correct.
\item
It is clear that the main for-loop maintains the invariant $M = \ord_N(\alpha)$
(thanks to lines \ref{line:alpha-M-init}, \ref{line:m-search}, \ref{line:lcm} and \ref{line:alpha-M-update}),
so if the algorithm terminates in \ref{line:M-large}, the output is correct.
\item
If $N$ is prime and sufficiently large, and if $D > N^{9/10}$, then in some iteration the algorithm will terminate correctly in line \ref{line:m-large}.
Indeed, Lemma \ref{lem:burgess} implies that the least primitive root modulo $N$ satisfies $g(N) < N^{3/10}$ for sufficiently large $N$.
If $D > N^{9/10}$ then $B > N^{3/10}$,
so line \ref{line:m-large} is guaranteed to find a primitive root,
which has order $N-1 > D$.
\end{itemize}

Assume now that the algorithm has reached line \ref{line:compute-Z}.
It remains to prove that the loop in lines \ref{line:k-loop}--\ref{line:d-test} will succeed in finding a nontrivial divisor $d \divides N$.
For the rest of this section, the symbol $M$ refers to the value of $M$ at the point that line \ref{line:compute-Z} is reached.
Note that $M \leq D$ due to line \ref{line:M-large}.

We claim that every prime divisor $p$ of $N$ satisfies
\begin{gather}
\label{eq:p-congruence}
p \equiv 1 \tpmod M, \\
\label{eq:Psi-p-bound}
\Psi(p, B) \leq M.
\end{gather}
Since no factors were found in line \ref{line:gcds} in any iteration of the for-loop,
\Cref{lem:porder} implies that $\ord_p(\beta) = m$ at each iteration,
and hence that $m \divides p-1$.
Since $M$ is the LCM of these values of $m$, we have $M \divides p-1$, i.e., \eqref{eq:p-congruence} holds.
Next, line~\ref{line:beta-divides-N} ensures that $p > B$,
and lines \ref{line:beta-to-M}, \ref{line:m-search} and \ref{line:lcm} ensure that $\beta^M \equiv 1 \pmod N$ for all positive integers $\beta \leq B$.
It follows that $n^M \equiv 1 \pmod p$ for all positive integers $n \leq p$ that are $B$-smooth.
The number of such integers $n$ is by definition $\Psi(p,B)$.
On the other hand, since $p$ is prime there are at most $M$ solutions to the congruence $n^M \equiv 1 \pmod p$. Thus \eqref{eq:Psi-p-bound} must hold.

Our main task is now to use \eqref{eq:Psi-p-bound} to show that
\begin{equation}
\label{eq:p-bound}
p \leq Z,
\end{equation}
where $Z$ is defined as in line \ref{line:compute-Z}.
Note that the definition of $\tilde Z$ makes sense, as both $M$ and $B$ are bounded well away from zero; for example $M \geq \Psi(p, B) \geq B \geq 4$ for sufficiently large $D$.

To prove \eqref{eq:p-bound}, it is enough to show that
\begin{equation}
\label{eq:Psi-tildeZ-bound}
\Psi(\tilde Z, B) > M.
\end{equation}
For if instead we had $p > Z$, then since $\Psi(x, B)$ is non-decreasing in $x$, we would have $\Psi(p, B) \geq \Psi(Z, B) \geq \Psi(\tilde Z, B) > M$, but this contradicts \eqref{eq:Psi-p-bound}.

Applying \Cref{lem:smooth} to $\Psi(\tilde Z, B)$
(this is valid as $\tilde Z > 2M > B \geq 4$) yields
\[
\Psi(\tilde Z, B) \geq \frac{\tilde Z}{(\log \tilde Z)^{\log \tilde Z / \log B}}.
\]
Thus to prove \eqref{eq:Psi-tildeZ-bound} it suffices to show that
\[
\frac{\tilde Z}{(\log \tilde Z)^{\log \tilde Z / \log B}} > M.
\]
In fact, we will prove the more precise statement that
\begin{equation}
\label{eq:tildeZ-interval}
M < \frac{\tilde Z}{(\log \tilde Z)^{\log \tilde Z / \log B}} < 4M.
\end{equation}

To prove \eqref{eq:tildeZ-interval}, let
\begin{equation}
\label{eq:z-formula}
z \coloneqq \log \tilde Z = \left(1 + \frac{\log \log 2M}{-1 + \log B}\right) \log 2M.
\end{equation}
Since $M \leq D$ and $\log B \asymp \log D$ we have
\begin{equation}
\label{eq:loglogM-logB-bound}
\frac{\log \log 2M}{-1 + \log B} \ll \frac{\log \log D}{\log D}.
\end{equation}
So by using the estimate $\log(1+x) = x + O(x^2)$, for sufficiently large $D$ we obtain
\begin{align*}
\log z
& = \frac{\log \log 2M}{-1 + \log B} + \log \log 2M + O\left(\frac{(\log \log D)^2}{\log^2 D}\right) \\
& = \frac{\log B}{-1 + \log B} \cdot \log \log 2M + O\left(\frac{(\log \log D)^2}{\log^2 D}\right).
\end{align*}
Thus
\[
1 - \frac{\log z}{\log B} = 1 - \frac{\log \log 2M}{-1 + \log B} + O\left(\frac{(\log \log D)^2}{\log^3 D}\right),
\]
and it follows that
\begin{align*}
z \left(1 - \frac{\log z}{\log B}\right)
& = \left(1 + \frac{\log \log 2M}{-1 + \log B}\right) \left(1 - \frac{\log \log 2M}{-1 + \log B} + O\left(\frac{(\log \log D)^2}{\log^3 D}\right)\right) \log 2M \\
& = \left(1 - \frac{(\log \log 2M)^2}{(-1 + \log B)^2} + O\left(\frac{(\log \log D)^2}{\log^3 D}\right)\right) \log 2M \\
& = \left(1 + O\left(\frac{(\log \log D)^2}{\log^2 D}\right)\right) \log 2M \\
& = \log 2M + O\left(\frac{(\log \log D)^2}{\log D}\right).
\end{align*}
For sufficiently large $D$, this implies that
\[
\log M < z \left(1 - \frac{\log z}{\log B}\right) < \log 4M,
\]
which is equivalent to \eqref{eq:tildeZ-interval}.
This completes the proof of \eqref{eq:p-bound}.

At this point we know that every prime $p$ dividing $N$ satisfies \eqref{eq:p-congruence} and \eqref{eq:p-bound}.
The final loop in lines \ref{line:k-loop}--\ref{line:d-test} simply tests every possible candidate.
If $N$ is composite, the loop will correctly terminate when it reaches the smallest prime divisor of $N$.
(If we let it run to completion, it will actually find all prime divisors of $N$.)

Now suppose that $N$ is prime.
As shown earlier, if $D > N^{9/10}$, then the algorithm would have already terminated in line \ref{line:m-large}.
Assume therefore that $D \leq N^{9/10}$.
From \eqref{eq:loglogM-logB-bound} we see that for large enough $D$ we have
$z \leq \frac{19}{18} \log 2M \leq \frac{19}{18} \log 2D$, and then
\[
Z < \tilde Z + 2 \leq (2D)^{19/18} + 2 \ll D^{19/18} \leq (N^{9/10})^{19/18} = N^{19/20},
\]
so $Z < N$ for sufficiently large $N$. But the only prime divisor of $N$ is $p = N$, which manifestly cannot satisfy \eqref{eq:p-bound} if $Z < N$. We have reached a contradiction; in other words, if $N$ is prime, the algorithm can never reach line \ref{line:compute-Z}, and must have already terminated (with a correct output) during the main for-loop.

\subsection{Complexity analysis}
\label{sec:complexity}

Lines \ref{line:lookup}--\ref{line:D-small} run in time $O(\log N)$.
After line \ref{line:D-small} we have $D \geq \log_2 N$,
so we may henceforth assume that $\log N \ll D$.
The computation of~$B$ in line \ref{line:compute-B} requires time
$O(\log N \log \log N) = O(\log N \log D)$.

Consider a single iteration of the for-loop.
If we have $\beta^M \equiv 1 \pmod N$ in line~\ref{line:beta-to-M}, the iteration is aborted and its complexity is $O(\log M \log N \log \log N)$.
Since $M \leq D$ for all of these iterations,
their total cost is
$O(B \log D \log N \log \log N) = O(D^{1/3} \log^2 D \log N)$,
which is negligible compared to \eqref{eq:main-complexity}.

We next consider the complexity of an iteration that reaches the search for $m$ in line \ref{line:m-search}.
If $m>D$, \Cref{lem:bsgs} implies that the algorithm terminates in line \ref{line:m-large} within the cost given by \eqref{eq:main-complexity}.
We hence assume that $m \leq D$. The complexity of line \ref{line:m-search} is then
\begin{equation} 
\label{eq:iter}
O\left(\frac{m^{1/2}}{(\log \log m)^{1/2}} \log D \log N\right).
\end{equation}
In line \ref{line:m-factor},
we find all divisors of $m$ via trial division in time
$O(m^{1/2}\log m \log\log m)$,
which is negligible compared to \eqref{eq:iter}.
In line \ref{line:gcds}, the cost of computing each power $\beta^{m/r} \pmod N$ is $O(\log m \log N \log \log N)$, and each GCD requires time $O(\log N \, (\log \log N)^2)$.
The number of primes $r$ dividing $m$ is $O(\log m)$,
so the cost of lines~\ref{line:r-loop}--\ref{line:gcds} is at most 
$O(\log^2 m \log N \, (\log \log N)^2) = O(\log^4 D \log N)$.
The total over all iterations is
$O(B \log^4 D \log N) = O(D^{1/3} \log^4 D \log N)$
which is negligible compared to \eqref{eq:main-complexity}.
According to \Cref{lem:compord},
line \ref{line:lcm} runs in time
$O((\log m +\log M) \log N \log \log N) = O((\log m + \log M) \log D \log N)$.
The first summand is bounded by \eqref{eq:iter}.
Since $M\leq D$, the combined cost of the second summand over all iterations is
$O(B \log^2 D \log N)=O(D^{1/3}\log^2 D \log N)$, which is negligible compared to 
\eqref{eq:main-complexity}.
We have now shown that in any iteration of the for-loop
that reaches line~\ref{line:m-search},
the cost of each step is either bounded by \eqref{eq:iter},
or is negligible when summed over all iterations.
It thus remains to estimate the sum of \eqref{eq:iter} over all iterations. 
	
Since we have assumed that $m \leq D$, for each such iteration we must have
\begin{equation}
\label{eq:m-interval}
\frac{D}{2^{l}} < m \leq \frac{D}{2^{l-1}}
\end{equation}
for some $l \in \{1, 2, \ldots, \lceil \log_2 D\rceil\}$. We claim that for each possible $l$, the number of iterations for which \eqref{eq:m-interval} holds with the given $l$ is at most $l+1$. To see this, first observe that line~\ref{line:beta-to-M} ensures that $m \ndivides M$ and hence that $M$ increases by a factor of at least 2 in each iteration that reaches line \ref{line:m-search}. On the first iteration that \eqref{eq:m-interval} occurs for a given $l$, we certainly have $M > D/2^l$ after line \ref{line:alpha-M-update}. After this point, the inequality $M \leq D$ can hold for at most $l$ more iterations, because line \ref{line:M-large} terminates the algorithm as soon as $M > D$ occurs. Therefore \eqref{eq:m-interval} can occur at most $l+1$ times as claimed.

Summing the total complexity over powers of two, we obtain
\begin{multline*}
\sum_{l=1}^{\lceil\log_2 D\rceil}
	\frac{(l+1)(D/2^{l-1})^{1/2}}{(\log\log(D/2^{l-1}))^{1/2}} \log D \log N \\
\ll D^{1/2} \log D \log N \sum_{l=1}^{\lceil\log_2 D\rceil}
	\frac{l/2^{l/2}}{(\log\log(D/2^{l-1}))^{1/2}}.
\end{multline*}
To estimate the inner sum, observe that
\begin{align*}
& \sum_{l=1}^{\lceil\log_2 D\rceil} \frac{l/2^{l/2}}{(\log\log(D/2^{l-1}))^{1/2}}\\
& \hspace{2em} \ll \sum_{l=1}^{\lceil(\log_2 D)/2\rceil-1} \frac{l/2^{l/2}}{(\log\log(D/2^{l-1}))^{1/2}} + \sum_{l=\lceil(\log_2 D)/2\rceil}^{\lceil\log_2 D\rceil} l/2^{l/2} \\
& \hspace{2em} \ll \sum_{l=1}^{\lceil(\log_2 D)/2\rceil-1} \frac{l/2^{l/2}}{(\log\log (D^{1/2}))^{1/2}} + D^{-1/4+o(1)} \ll (\log\log D)^{-1/2}.
\end{align*}
Thus the total contribution of \eqref{eq:iter} over all iterations is bounded by \eqref{eq:main-complexity}.

Finally we consider the complexity of lines \ref{line:compute-Z}--\ref{line:d-test}.
From \eqref{eq:loglogM-logB-bound} we have
$z \ll \log M \leq \log D \leq \log N$ and hence $\tilde Z < N^{O(1)}$.
Using the formula \eqref{eq:z-formula},
we may approximate $z$ and then $\tilde Z = e^z$ with
$O(\log N)$ bits of precision in time
$O(\log N \, (\log \log N)^2) = O(\log^2 D \log N)$.
Thus the cost of computing $Z$ in line \ref{line:compute-Z} is
$O(\log^2 D \log N)$.
The cost of the loop in lines \ref{line:k-loop}--\ref{line:d-test} is
$O((Z/M) \log N \log \log N) = O((Z/M) \log D \log N)$.
We have
\[
\frac{Z}{M} \ll \frac{\tilde Z}{2M} = (\log 2M)^{(\log 2M) / (-1 + \log B)}.
\]
Taking logarithms,
\[
\log(Z/M) < \frac{\log 2M \log \log 2M}{-1 + \log B} + O(1) \ll \frac{\log D \log \log D}{\log D} = \log \log D.
\]
This implies that $Z/M < (\log D)^{O(1)}$, so the cost of the loop is
\[
(\log D)^{O(1)} \log N,
\]
which is negligible compared to \eqref{eq:main-complexity}.

\bibliographystyle{amsalpha}
\bibliography{main}

@article {W-leastprimroot,
author = {Wang, Y.},
title = {On the least primitive root of a prime},
journal = {Scientia Sinica},
volume = {10},
pages = {1-14},
year = {1961},
}

@inproceedings{S-primroots,
author = {Shoup, V.},
title = {Searching for primitive roots in finite fields},
year = {1990},
isbn = {0897913612},
publisher = {Association for Computing Machinery},
address = {New York, NY, USA},
url = {https://doi.org/10.1145/100216.100293},
doi = {10.1145/100216.100293},
booktitle = {Proceedings of the Twenty-Second Annual ACM Symposium on Theory of Computing},
pages = {546-554},
numpages = {9},
location = {Baltimore, Maryland, USA},
series = {STOC '90}
}

@unpublished{GFHP-factoringlattices,
      author = {Gao, Y. and Feng, Y. and Hu, H. and Pan, Y.},
      title = {On factoring and power divisor problems via rank-3 lattices and the second vector},
      note = {Cryptology {ePrint} Archive, \url{https://eprint.iacr.org/2025/1004}, to appear in Mathematics of Computation, \url{https://doi.org/10.1090/mcom/4188}},
      year = {2025},
}

@inproceedings{OV-highorder,
author = {Oznovich, Z. and Volk, B. L.},
title = {On Deterministically Finding an Element of High Order Modulo a Composite},
year={2026},	
publisher = {Society for Industrial and Applied Mathematics},
address   = {},
booktitle = {Proceedings of the 2026 Annual ACM-SIAM Symposium on Discrete Algorithms (SODA)},
chapter = {},
pages = {6379-6391},
doi = {10.1137/1.9781611978971.229},
URL = {https://epubs.siam.org/doi/abs/10.1137/1.9781611978971.229},
eprint = {https://epubs.siam.org/doi/pdf/10.1137/1.9781611978971.229}
}

@article{BM-prng,
author = {Blum, M. and Micali, S.},
title = {How to Generate Cryptographically Strong Sequences of Pseudorandom Bits},
journal = {SIAM Journal on Computing},
volume = {13},
number = {4},
pages = {850-864},
year = {1984},
doi = {10.1137/0213053},
}

@article{E-elgamal,
  author={Elgamal, T.},
  journal={IEEE Transactions on Information Theory}, 
  title={A public key cryptosystem and a signature scheme based on discrete logarithms}, 
  year={1985},
  volume={31},
  number={4},
  pages={469-472},
  doi={10.1109/TIT.1985.1057074}}

@article{DH-dlp,
  author    = {Diffie, W. and Hellman, M. E.},
  title     = {New Directions in Cryptography},
  journal   = {IEEE Transactions on Information Theory},
  volume    = {22},
  number    = {6},
  pages     = {644--654},
  year      = {1976},
  url       = {https://ee.stanford.edu/~hellman/publications/24.pdf}
}

@article {B-leastprimitiveroot,
    AUTHOR = {Burgess, D. A.},
     TITLE = {On character sums and primitive roots},
   JOURNAL = {Proc. London Math. Soc. (3)},
  FJOURNAL = {Proceedings of the London Mathematical Society. Third Series},
    VOLUME = {12},
      YEAR = {1962},
     PAGES = {179--192},
      ISSN = {0024-6115,1460-244X},
   MRCLASS = {10.41},
  MRNUMBER = {132732},
MRREVIEWER = {L.\ Carlitz},
       DOI = {10.1112/plms/s3-12.1.179},
       URL = {https://doi.org/10.1112/plms/s3-12.1.179},
}

@article{HH-lehmangeneralization,
  title={A generalization of {L}ehman's method},
  author={Hales, J. and Hiary, G.},
  journal={The Ramanujan Journal},
  volume={65},
  number={4},
  pages={1773--1790},
  year={2024},
  publisher={Springer}
}

@article {Hit-BSGS,
    AUTHOR = {Hittmeir, M.},
     TITLE = {A babystep-giantstep method for faster deterministic integer
              factorization},
   JOURNAL = {Math. Comp.},
  FJOURNAL = {Mathematics of Computation},
    VOLUME = {87},
      YEAR = {2018},
    NUMBER = {314},
     PAGES = {2915--2935},
      ISSN = {0025-5718},
   MRCLASS = {11A51 (11Y05 11Y16 68Q25)},
  MRNUMBER = {3834692},
MRREVIEWER = {Ra\'{u}l Dur\'{a}n D\'{\i}az},
       DOI = {10.1090/mcom/3313},
       URL = {https://doi.org/10.1090/mcom/3313},
}

@article {Hit-timespace,
  author = {Hittmeir, M.},
  title={A time-space tradeoff for {L}ehman's deterministic integer factorization method},
  journal={Mathematics of Computation},
  volume={90},
  number={330},
  pages={1999--2010},
  year={2021}
}

@article {Har-onefifth,
  title={An exponent one-fifth algorithm for deterministic integer factorisation},
  author={Harvey, D.},
  journal={Mathematics of Computation},
  volume={90},
  number={332},
  pages={2937--2950},
  year={2021}
}

@article {HH-onefifthloglog,
  title={A log-log speedup for exponent one-fifth deterministic integer factorisation},
  author={Harvey, D. and Hittmeir, M.},
  journal={Mathematics of Computation},
  volume={91},
  number={335},
  pages={1367--1379},
  year={2022}
}

@book {Pap-complexity,
    AUTHOR = {Papadimitriou, C. H.},
     TITLE = {Computational {C}omplexity},
 PUBLISHER = {Addison-Wesley Publishing Company},
   ADDRESS = {Reading, MA},
      YEAR = {1994},
     PAGES = {xvi+523},
      ISBN = {0-201-53082-1},
   MRCLASS = {68Q15 (68-02 68Q25)},
  MRNUMBER = {1251285 (95f:68082)},
MRREVIEWER = {Johan H{\aa}stad},
}

@article {HvdH-nlogn,
    AUTHOR = {Harvey, D. and van der Hoeven, J.},
     TITLE = {Integer multiplication in time {$O(n \log n)$}},
   JOURNAL = {Ann. of Math. (2)},
  FJOURNAL = {Annals of Mathematics. Second Series},
    VOLUME = {193},
      YEAR = {2021},
    NUMBER = {2},
     PAGES = {563--617},
      ISSN = {0003-486X},
   MRCLASS = {11Y16 (68W30)},
  MRNUMBER = {4224716},
       DOI = {10.4007/annals.2021.193.2.4},
       URL = {https://doi.org/10.4007/annals.2021.193.2.4},
}

@book {vzGG-compalg3,
    AUTHOR = {von zur Gathen, J. and Gerhard, J.},
     TITLE = {Modern {C}omputer {A}lgebra},
   EDITION = {3rd},
 PUBLISHER = {Cambridge University Press, Cambridge},
      YEAR = {2013},
     PAGES = {xiv+795},
      ISBN = {978-1-107-03903-2},
   MRCLASS = {68W30 (11Y05 11Y11 13Pxx)},
  MRNUMBER = {3087522},
       DOI = {10.1017/CBO9781139856065},
       URL = {http://dx.doi.org/10.1017/CBO9781139856065},
}

@book {BZ-mca,
    AUTHOR = {Brent, R. P. and Zimmermann, P.},
     TITLE = {Modern {C}omputer {A}rithmetic},
    SERIES = {Cambridge Monographs on Applied and Computational Mathematics},
    VOLUME = {18},
 PUBLISHER = {Cambridge University Press, Cambridge},
      YEAR = {2011},
     PAGES = {xvi+221},
      ISBN = {978-0-521-19469-3},
   MRCLASS = {65Y04},
  MRNUMBER = {2760886},
}

@article {Str-factoring,
    AUTHOR = {Strassen, V.},
     TITLE = {Einige {R}esultate \"uber {B}erechnungskomplexit\"at},
   JOURNAL = {Jber. Deutsch. Math.-Verein.},
  FJOURNAL = {Jahresbericht der Deutschen Mathematiker-Vereinigung},
    VOLUME = {78},
      YEAR = {1976/77},
    NUMBER = {1},
     PAGES = {1--8},
      ISSN = {0012-0456},
   MRCLASS = {68A20},
  MRNUMBER = {0438807 (55 \#11713)},
MRREVIEWER = {Wilfried Brauer},
}

@article {Cop-lowexponent,
    AUTHOR = {Coppersmith, D.},
     TITLE = {Small solutions to polynomial equations, and low exponent
              {RSA} vulnerabilities},
   JOURNAL = {J. Cryptology},
  FJOURNAL = {Journal of Cryptology. The Journal of the International
              Association for Cryptologic Research},
    VOLUME = {10},
      YEAR = {1997},
    NUMBER = {4},
     PAGES = {233--260},
      ISSN = {0933-2790},
   MRCLASS = {94A60 (11Y05)},
  MRNUMBER = {1476612},
MRREVIEWER = {Eric Bach},
       DOI = {10.1007/s001459900030},
       URL = {https://doi.org/10.1007/s001459900030},
}

@inproceedings {Sha-classnumber,
    AUTHOR = {Shanks, D.},
     TITLE = {Class number, a theory of factorization, and genera},
 BOOKTITLE = {1969 {N}umber {T}heory {I}nstitute ({P}roc. {S}ympos. {P}ure
              {M}ath., {V}ol. {XX}, {S}tate {U}niv. {N}ew {Y}ork, {S}tony
              {B}rook, {N}.{Y}., 1969)},
     PAGES = {415--440},
      YEAR = {1971},
   MRCLASS = {10C05 (10H10 12A25 12A50)},
  MRNUMBER = {0316385},
MRREVIEWER = {W. Narkiewicz},
}

@article{Pol-factoring, 
author={Pollard, J. M.},
title={Theorems on factorization and primality testing}, 
volume={76}, 
DOI={10.1017/S0305004100049252}, 
number={3}, 
journal={Mathematical Proceedings of the Cambridge Philosophical Society}, 
year={1974}, 
pages={521–528}}

@book {Wag-joy,
    AUTHOR = {Wagstaff, Jr., S. S.},
     TITLE = {The {J}oy of {F}actoring},
    SERIES = {Student Mathematical Library},
    VOLUME = {68},
 PUBLISHER = {American Mathematical Society, Providence, RI},
      YEAR = {2013},
     PAGES = {xiv+293},
      ISBN = {978-1-4704-1048-3},
   MRCLASS = {11Y05 (11A41 11Y11 11Y16)},
  MRNUMBER = {3135977},
MRREVIEWER = {Carl Pomerance},
       DOI = {10.1090/stml/068},
       URL = {https://doi.org/10.1090/stml/068},
}

@book {Rie-factorization,
    AUTHOR = {Riesel, H.},
     TITLE = {Prime {N}umbers and {C}omputer {M}ethods for {F}actorization},
    SERIES = {Progress in Mathematics},
    VOLUME = {126},
   EDITION = {Second},
 PUBLISHER = {Birkh\"{a}user Boston, Inc., Boston, MA},
      YEAR = {1994},
     PAGES = {xvi+464},
      ISBN = {0-8176-3743-5},
   MRCLASS = {11Y05 (11A41 11A51 68Q40 94A60)},
  MRNUMBER = {1292250},
       DOI = {10.1007/978-1-4612-0251-6},
       URL = {https://doi.org/10.1007/978-1-4612-0251-6},
}

@book {SutherlandPHD,
    AUTHOR = {Sutherland, A.V.},
     TITLE = {Order computations in generic groups},
    SERIES = {Thesis (Ph.D.)--Massachusetts Institute of Technology},
    VOLUME = {},
   EDITION = {},
 PUBLISHER = {ProQuest LLC, Ann Arbor},
      YEAR = {2007},
     PAGES = {},
      ISBN = {},
   MRCLASS = {},
  MRNUMBER = {},
MRREVIEWER = {},
       DOI = {},
       URL = {http://hdl.handle.net/1721.1/38881},
}

@incollection {Konyagin2013,
    AUTHOR = {Konyagin, S. and Pomerance, C.},
     TITLE = {On primes recognizable in deterministic polynomial time},
 BOOKTITLE = {The mathematics of {P}aul {E}rd\H os, {I}},
    SERIES = {Algorithms Combin.},
    VOLUME = {13},
     PAGES = {176--198},
 PUBLISHER = {Springer, Berlin},
      YEAR = {1997},
      ISBN = {3-540-61032-4},
   MRCLASS = {11Y11 (11Y16)},
  MRNUMBER = {1425185},
MRREVIEWER = {Wieb\ Bosma},
       DOI = {10.1007/978-3-642-60408-9\_15},
       URL = {https://doi.org/10.1007/978-3-642-60408-9_15},
}

@book {AHU-algorithms,
   AUTHOR = {Aho, A. V. and Hopcroft, J. E. and Ullman, J. D.},
   TITLE = {The {D}esign and {A}nalysis of {C}omputer {A}lgorithms},
   SERIES = {Addison-Wesley Series in Computer Science and Information
   Processing},
   NOTE = {Second printing},
   PUBLISHER = {Addison-Wesley Publishing Co., Reading,
   Mass.-London-Amsterdam},
   YEAR = {1975},
   PAGES = {x+470},
   MRCLASS = {68A10 (68A20)},
   MRNUMBER = {413592},
   MRREVIEWER = {M.\ Tetruashvili},
}

@article {AKS-primes,
    AUTHOR = {Agrawal, M. and Kayal, N. and Saxena, N.},
     TITLE = {P{RIMES} is in {P}},
   JOURNAL = {Ann. of Math. (2)},
  FJOURNAL = {Annals of Mathematics. Second Series},
    VOLUME = {160},
      YEAR = {2004},
    NUMBER = {2},
     PAGES = {781--793},
      ISSN = {0003-486X,1939-8980},
   MRCLASS = {11Y11 (11A51 68Q15 68Q25)},
  MRNUMBER = {2123939},
MRREVIEWER = {Carl\ Pomerance},
       DOI = {10.4007/annals.2004.160.781},
       URL = {https://doi.org/10.4007/annals.2004.160.781},
}

@article {Shp-primitive,
    AUTHOR = {Shparlinski, I.},
     TITLE = {On finding primitive roots in finite fields},
   JOURNAL = {Theoret. Comput. Sci.},
  FJOURNAL = {Theoretical Computer Science},
    VOLUME = {157},
      YEAR = {1996},
    NUMBER = {2},
     PAGES = {273--275},
      ISSN = {0304-3975},
     CODEN = {TCSDI},
   MRCLASS = {11Y16 (11T99 68Q40)},
  MRNUMBER = {1389773},
       DOI = {10.1016/0304-3975(95)00164-6},
       URL = {http://dx.doi.org/10.1016/0304-3975(95)00164-6},
}

@unpublished{Nir2026,
      title={Deterministically finding an element of large order in $\mathbb{Z}_N^*$}, 
      author={Nir, I.},
      year={2026},
      eprint={2605.09592},
      archivePrefix={arXiv},
      primaryClass={cs.DS},
      note={preprint at \url{https://arxiv.org/abs/2605.09592}}
}

\end{document}